\documentclass[reqno]{amsart}
\usepackage{amsmath}
\usepackage{amsfonts,amscd, amssymb}
\usepackage{hyperref}
\usepackage{xcolor}
\usepackage{tikz-cd}
\usepackage{amsthm,latexsym,euscript,amscd}

\usetikzlibrary{matrix}
\newtheorem{theorem}{Theorem}
\newtheorem{lemma}{Lemma}
\newtheorem{corollary}{Corollary}
\newtheorem{proposition}{Proposition}
\newtheorem{remark}{Remark}
\theoremstyle{definition}
\newtheorem{definition}{Definition}

\numberwithin{equation}{section} 
\makeatother

\DeclareMathOperator{\Idb}{{\mathbb I}}

\DeclareMathOperator{\Rdb}{{\mathbb R}}

\DeclareMathOperator{\Al}{{\mathcal A}}

\newcommand{\Sp}[1]{\operatorname{Sp}(#1)}

\newcommand{\norm}[1]{\left\Vert#1\right\Vert}

\numberwithin{equation}{section}

\begin{document}
\title[Operator means in JB-algebras]{Operator means in JB-algebras}

\author{Shuzhou Wang}
\address{Department of Mathematics, University of Georgia, Athens, GA, 30602}
\email{szwang@uga.edu}
\author{Zhenhua Wang}
\address{Department of Mathematics, University of Georgia, Athens, GA, 30602}
\email{ezrawang@uga.edu}

%    General info
\date{}
\begin{abstract}
In this paper, the notion of operator means in the setting of JB-algebras is  
introduced and their properties are studied.  
Many identities and inequalities are established, most of them 
have origins from  operators on Hilbert space 
but they have different forms and connotations, and their proofs require 
techniques in JB-algebras. 
\end{abstract}
\maketitle

\section{Introduction}

Arising in the work of Jordan, von Neumann, and Wigner on the axiomatic foundations of quantum mechanics, finite dimensional Jordan algebras were investigated first in \cite{Jordan1934}. Later the infinite dimensional case were studied by von Neumann in \cite{Neumann1936algebraic}. Jordan subalgebras of selfadjoint part of operators on a Hilbert space were initiated by Segal \cite{Segal1947}, and later studied by Effros and St{\o}rmer \cite{effros1967jordan}, St{\o}rmer \cite{stormer1965, stormer1966, stormer1968} and Topping \cite{topping1965jordan}, etc. The theory of JB-algebras was inaugurated by Alfsen, Shultz, and St{\o}rmer \cite{ALFSEN197811} and later considered by many others. As a motivation for this line of research, 
the observables in a quantum system constitute a JB-algebra, which is non-associative, therefore JB-algebras were considered as natural objects of study for quantum mechanics. 
JB-algebras also have many powerful applications in other fields, such as analysis, geometry, 
operator theory, etc; more information on these can be found in 
\cite{Chu,Upmeier, Upmeier1}.  
 
In a different direction, the theory of operator means started from the notion of parallel addition for two positive matrices introduced by Anderson and Duffin \cite{ANDERSON1969576} as a tool to study electrical network synthesis. Later this notion was extended to positive operators on a Hilbert space by Anderson and Trapp \cite{ANDERSON197560} to solve maximization problem in electrical network theory. On the other hand, in 1975, the geometric mean for two positive operators on Hilbert space was considered by Pusz and Woronowicz \cite{PUSZ1975159}. The general theory of operator means was initiated by Ando \cite{Ando1978topics} and established by Ando \cite{ANDO1979203, ANDO198331}, Kubo and Ando \cite{Kubo1979}, Fujii \cite{fujii1978arithmetico, fujii1979geometric, fujii1979some}, and many others. It turned out operator means, especially arithmetic mean, harmonic mean, and geometric mean have significant impact on operator theory \cite{ANDO1979203}. The interest in operator means and related objects has been exclusively restricted to the context of operators on a Hilbert space. As far as we know, there is no paper in the literature that discusses operator means in JB-algebras. 

In the present paper, we introduce the notion of operator means and study their properties 
in the setting of JB-algebras. These algebras include JC-algebras, therefore selfadjoint operators on a Hilbert space, as a special case, but there are JB-algebras, 
such as the Albert algebra, that are not JC-algebras. Such algebras are nonassociative, therefore 
Hilbert space operator techniques do not apply. 
Many identities and inequalities for JB-algebras are established in this paper. Though most of them 
have origins from operators on Hilbert space, they have different forms and their proofs require 
techniques in JB-algebras.

\section{Preliminaries} 

In this section, we give some background on JB-algebras and fix the notation. 

\begin{definition}\label{DefJa}
A {\bf Jordan algebra} $\Al$ over real number is a vector space $\Al$ over $\Rdb$ equipped with a  bilinear product $\circ$ that satisfies the following identities:
$$a\circ b =b\circ b, \,\ \,\ (a^2\circ b)\circ a=a^2\circ (b\circ a).$$
Any associative algebra $\Al$ has an underlying Jordan algebra structure with Jordan product given by 
$$a\circ b=(ab+ba)/2.$$
Jordan suablgebra of such underlying Jordan algebras is called {\bf special}. 
\end{definition}

As the important example in physics,  
$B(H)_{sa}$, the set of bounded self adjoint operators on a Hilbert space $H$,    
is a special Jordan algebra. Note that $B(H)_{sa}$ is not an associative algebra.

\begin{definition}
A concrete {\bf JC-algebra} $\Al$ is a norm-closed Jordan subalgebra of $B(H)_{sa}$.
\end{definition}
 \begin{definition}\label{DefJB}
 A {\bf JB-algebra} is a Jordan algebra $\Al$ over $\Rdb$ with a complete norm satisfying the following conditions for $A, B\in \Al:$ 
 \begin{align*}
 	\norm{A\circ B}\leq \norm{A}\norm{B},~~\norm{A^2}=\norm{A}^2,~~\mbox{and}~~\norm{A^2}\leq \norm{A^2+B^2}.	
 \end{align*}	
 \end{definition}

A JC-algebra is a JB-algebra, but the converse is not true. 
For example, the Albert algebra is a JB-algebra but not a JC-algebra, cf. \cite[Theorem 4.6]{Alfsen2003}.  

\begin{definition}
Let $\Al$ be a unital JB-algebra. 
We say $A\in \Al$ is {\bf invertible} if there exists $B\in \Al,$ which is called {\bf Jordan inverse} of $A,$ such that  
\begin{align*}
	A\circ B=I \quad \mbox{and}\quad A^2\circ B=A.	
\end{align*}
 The {\bf spectrum} of $A$ is defined by 
 \begin{align*}
 \Sp{A}:=\{\lambda\in \Rdb| A-\lambda I\,\ \text{ is not invertible in} \Al \}.	
 \end{align*}
If $\Sp{A}\subset [0,\infty),$ we say $A$ is {\bf positive}, and write $A \geq 0$. 	
\end{definition}

 \begin{definition}
Let $\Al$ be a unital JB-algebra and $A, B \in \Al$.  We define a map $U_A$ on $\Al$ by
\begin{align}\label{JI}
	U_{A}B:=\{ABA\}:= 2(A\circ B)\circ A -A^2\circ B.
\end{align}
 \end{definition}
It follows from (\ref{JI}) that $U_A$ is linear, in particular, 
\begin{align}\label{LJI}
	U_A(B-C)=\{ABA\}-\{ACA\}.
\end{align} 
Note that $ABA$ is meaningless unless $\Al$ is special, in which case $\{ABA\}=ABA.$ 
The following proposition will be used repeatedly in this paper.

\begin{proposition}\cite[Lemma 1.23-1.25]{Alfsen2003} \label{3inv}
Let $\Al$ be a unital JB-algebra and $A, B$ be two elements in $\Al$.
\begin{enumerate}
	\item If $B$ is positive, then $U_A(B)=\{ABA\}\geq 0.$ 
	\item If $A, B$ are invertible, then $\{ABA\}$ is invertible with inverse $\{A^{-1}B^{-1}A^{-1}\}.$
	\item If $A$ is invertible, then $U_A$ has a bounded inverse $U_{A^{-1}}.$

\end{enumerate}
	
\end{proposition}

For an element $A$ in $\Al$ and a continuous function $f$ on the spectrum of $A$, 
$f(A)$ is defined by functional calculus in JB-algebras (see e.g. \cite[Proposition 1.21]{Alfsen2003}).
\begin{definition}
Let $f$ is a real valued continuous function $f$ on $\Rdb.$ 
\begin{enumerate}
	\item $f$ is said to be 
	{\bf operator monotone (increasing)} on a JB-algebra $\Al$ if 
	$0\leq A\leq B$ implies $f(A)\leq f(B)$.
	\item $f$ is {\bf operator convex} if for any $\lambda \in [0, 1]$ and $A, B\geq 0,$
	$$f((1-\lambda)A+\lambda B)\leq (1-\lambda) f(A)+\lambda f(B).$$
	We say that $f$ is {\bf operator concave} if $-f$ is operator convex.
\end{enumerate}
\end{definition}

\section{Nonassociative perspective function}
In this section, the nonassociative perspective function is introduced and some of its properties are also discussed. 
\begin{definition}
Let $f$ and $h$ be real continuous function on a closed interval $\Idb$ with $h>0$ and let $A, B$ be two elements in a unital JB-algebra $\Al$ with spectra contained in $\Idb$ and $\Sp{\{h(B)^{-\frac{1}{2}}Ah(B)^{-\frac{1}{2}}\}}\subset \Idb$. The {\bf nonassociative perspective function}
$P_{f\triangle h}(A,B)$ of two variables $A$ and $B$ associated to $f$ and $h$ is defined by
\begin{align}\label{dnapf}
	P_{f\triangle h}(A,B)
	=\left\{h(B)^{\frac{1}{2}}f\left(\{h(B)^{-\frac{1}{2}}Ah(B)^{-\frac{1}{2}}\}\right)h(B)^{\frac{1}{2}} \right\}.	
\end{align}	
\end{definition}
\begin{remark} Though $\Al$ is commutative under the nonassociative product $\circ$, if $\Al$ is special, then our notion is precisely the noncommutative perspective we introduced in \cite{Wang2020}, which is a generalization of  \cite{ebadian2011perspectives} for noncommuative associative case. For commutative and associative case, 
	the notion was initiated by Effros \cite{effros2009matrix}, in which Effros presented a simple approach to the famous Lieb's concavity theorem \cite{LIEB1973267, Lieb19731938}.
\end{remark}

%\begin{proposition}\label{omnapf}
%Let $f$ and $h$ be real continuous function on a closed interval $\Idb$ with $f$ operator monotone and $h>0$. For positive invertible elements $B, C$ with $B\leq C$ and an element $A$ in a unital JB-algebra $\Al$ such that the spectra of $A,$ $\{h(A)^{-1/2}Bh(A)^{-1/2}\},$ and $\{h(A)^{-1/2}C h(A)^{-1/2}\}$ are contained in $\Idb$, 
%\begin{align}
%P_{f\triangle h}(B, A)\leq P_{f\triangle h}(C, A)	
%\end{align}
%\end{proposition}
%\begin{proof}
%Since $B\leq C,$ then by Theorem 1.25 \cite{Alfsen2003}
%\begin{align*}
%\{h(A)^{-\frac{1}{2}}Bh(A)^{-\frac{1}{2}}\}\leq 	\{h(A)^{-\frac{1}{2}}Ch(A)^{-\frac{1}{2}}\}
%\end{align*}
%If $f$ is operator monotone, then 
%\begin{align*}
%f\left(\{h(A)^{-\frac{1}{2}}Bh(A)^{-\frac{1}{2}}\}\right)\leq 	f\left(\{h(A)^{-\frac{1}{2}}Ch(A)^{-\frac{1}{2}}\}\right)
%\end{align*}
%Finally,
%\begin{align*}
%P_{f\triangle h}(B, A)\leq P_{f\triangle h}(C, A).	
%\end{align*}
%\end{proof}
\begin{proposition}\label{omnapf}
	Let $f$ and $h$ be real continuous function on a closed interval $\Idb$ with $f$ operator monotone and $h>0$. For positive elements $A, B$ with $A\leq B$ and an element $C$ in a unital JB-algebra $\Al$ such that the spectra of $C,$ $\{h(C)^{-1/2}Ah(C)^{-1/2}\},$ and $\{h(C)^{-1/2}B h(C)^{-1/2}\}$ are contained in $\Idb$, 
	\begin{align}
		P_{f\triangle h}(A, C)\leq P_{f\triangle h}(B, C).	
	\end{align}
\end{proposition}
Since $A\leq B,$ then by Proposition \ref{3inv}
\begin{align*}
	\{h(C)^{-\frac{1}{2}}Ah(C)^{-\frac{1}{2}}\}\leq 	\{h(C)^{-\frac{1}{2}}Bh(C)^{-\frac{1}{2}}\}.
\end{align*}
If $f$ is operator monotone, then 
\begin{align*}
	f\left(\{h(C)^{-\frac{1}{2}}Ah(C)^{-\frac{1}{2}}\}\right)\leq 	f\left(\{h(C)^{-\frac{1}{2}}Bh(C)^{-\frac{1}{2}}\}\right).
\end{align*}
Finally,
\begin{align*}
	P_{f\triangle h}(A, C)\leq P_{f\triangle h}(B, C).	
\end{align*}

\begin{remark}
By Shirshov-Cohen theorem for JB-algebras \cite[Theorem 7.2.2]{hanche1984jordan}, the JB-subalgebra generated by $A, C$ is a JC-algebra, so is the JB-subalgebra generated by $B, C.$ However, one cannot use this fact 
in the proof of Proposition \ref{omnapf} by reducing to JC-algebras since the JB-subalgebra generated $A, B, C$ usually is not a JC-algebra. The same situation repeatedly occurs in the rest of paper.  

On the other hand, a few results contained in this paper can be obtained by reducing to JC-algebras, but for consistency, we will only use techniques for JB-algebras.  
\end{remark}

\begin{proposition}\label{ocvnapf}
Let $f$ and $h$ be real continuous functions on $[0, \infty)$ with $f$ operator convex and $h>0$. Assume that $A_k \geq 0,  B$ are elements in a unital JB-algebra $\Al$, $k=1, 2$.  
 Then for any $0\leq \lambda \leq 1,$ 
\begin{align}
P_{f\triangle h}(\lambda A_1+(1-\lambda)A_2, B)\leq \lambda P_{f\triangle h}( A_1, B)+(1-\lambda)P_{f\triangle h}( A_2, B).	
\end{align}
\end{proposition}
\begin{proof}
Note that 
\begin{align*}
\left\{h(B)^{-\frac{1}{2}}\left(\lambda A_1+(1-\lambda)A_2\right)h(B)^{-\frac{1}{2}}\right\}=&\lambda \left\{h(B)^{-\frac{1}{2}}A_1h(B)^{-\frac{1}{2}}\right\}\\
&+(1-\lambda) \left\{h(B)^{-\frac{1}{2}}A_2h(B)^{-\frac{1}{2}}\right\}	.
\end{align*}
By the operator convexity of $f$,  
\begin{align*}
f\left(\left\{h(B)^{-\frac{1}{2}}\left(\lambda A_1+(1-\lambda)A_2\right)h(B)^{-\frac{1}{2}}\right\}\right)\leq& \lambda f\left(\left\{h(B)^{-\frac{1}{2}}A_1h(B)^{-\frac{1}{2}}\right\}\right)\\
&+(1-\lambda) f\left(\left\{h(B)^{-\frac{1}{2}}A_2h(B)^{-\frac{1}{2}}\right\}\right).
\end{align*}
Finally, by Proposition \ref{3inv},
\begin{align*}
P_{f\triangle h}(\lambda A_1+(1-\lambda)A_2, B)\leq \lambda P_{f\triangle h}( A_1, B)+(1-\lambda)P_{f\triangle h}( A_2, B).	
\end{align*}
\end{proof}
\begin{theorem}
	\label{tnapf} 
	Let $\Al$ be a unital JB-algebra.	
	Let $r, q$ and $h$ be real valued continuous functions on a closed interval $\Idb$ such that $h>0$ 
	and $r(x)\leq q(x)$. For elements $A$ 
	and $B$ in $\Al$ such that the spectra of $B$ and $\{h(B)^{-1/2}Ah(B)^{-1/2}\}$ are contained in $\Idb$, 
	\begin{equation}\label{inapf} 
		P_{r\triangle h}(A,B)\leq P_{q\triangle h}(A, B).
	\end{equation}
\end{theorem}

\begin{proof}
	Since the functional calculus is
	order preserving, cf. \cite{Alfsen2003}, 
	$$ r\left(\{h(B)^{-\frac{1}{2}}Ah(B)^{-\frac{1}{2}}\}\right)\leq q\left(\{h(B)^{-\frac{1}{2}}Ah(B)^{-\frac{1}{2}}\}\right).$$ 
	 It implies that 
	\begin{align}\label{inapfs1}
		&\left\{h(A)^{\frac{1}{2}}\left[(q-r)\left(\{h(A)^{-\frac{1}{2}}Bh(A)^{-\frac{1}{2}}\}\right)\right]h(A)^{\frac{1}{2}} \right\}\geq 0
	\end{align}  

	Combing (\ref{LJI}), (\ref{dnapf})and (\ref{inapfs1}) , we have
	\begin{align*}
		P_{r\triangle h}(B, A)\leq P_{q\triangle h}(B, A). 
	\end{align*}
\end{proof}

\begin{remark}
	In later sections, we do not explicitly mention interval $\Idb$ because the relevant functions 
	are continuous on $(0, \infty)$. 
\end{remark}

\section{Operator concavity and monotonicity}

\iffalse
%
%We recall that Riemman integral for Banach space valued functions 
%defined on a closed integral $[a, \ b]$ can be defined exactly the same way as 
%for real value functions on $[a, \ b]$ with absolute value for real numbers 
%replaced by then norm on the Banach space. The same is true for Riemann improper integral of Banach space valued functions defined on $(0, \ \infty)$. 
%% 
%\begin{definition} \label{RiemannInt}
%	Let $\mathcal{X}$ be a Banach space and $h(\alpha)$ is a 
%	$\mathcal{X}$-valued continuous function on $\alpha \in (0, \infty)$. 
%\end{definition}

\fi

\begin{definition} \label{RiemannInt}
Let $f_\alpha(x)$ be a real valued function on 
$(\alpha, x)  \in (0, \infty) \times [0, \infty)$ 
that is separately continuous with respect to  $\alpha$ and $x$.   
We say $f_\alpha(x)$ is {\bf uniformly Riemann integrable} on $ \alpha  \in (0, \infty)$ for $x$ on 
bounded and closed intervals if for any closed interval $[0, M] \subset [0, \infty)$, and   
$\varepsilon > 0$, there exist $1 > \delta> 0$ and $N>1$, such that for any positive numbers   
$\delta_1, \delta_2 \leq  \delta$, $N_1, N_2 \geq N$, partitions 
$\Delta_\beta $, $\Delta_\gamma$ of $[\delta, N]$ with 
$\max_{k ,l}\left\{|\Delta \beta_k|, |\Delta \gamma_l|\right\}<\delta $, and 
$\text{for all} \; \;  x \in [0,M],$
we have 
\begin{align}
&\left|\int_{\delta_1}^{\delta_2} f_\alpha(x) d \alpha \right| < \varepsilon/3, \; \; \; 
\left|\int_{N_1}^{N_2} f_\alpha(x) d \alpha \right| < \varepsilon/3, \; \; \;\\
&\left|\sum_k f_{\beta^*_k}(x) \Delta \beta_k  - 
\sum_l f_{\gamma^*_l}(x) \Delta \gamma_l  \right| < \varepsilon/3,  \label{Riemm2}\; \; \; 	
\end{align}
where $\beta^*_k \in [\beta_{k-1}, \beta_{k}]$ and 
$\gamma^*_l \in [\gamma_{l-1}, \gamma_{l}]$ are arbitrary.

\iffalse
separately continuous functions. Moreover, we say $f_\alpha(x)$ is continuous in $\alpha$ 
uniformly with respect to $x$ on closed bounded intervals if $[a, b]$ is any closed 
interval contained in $[0, \infty) $ such that for each fixed $\alpha_0,$ $f_{\alpha}(x)\to f_{\alpha_0}(x)$ uniformly as $\alpha\to \alpha_0.$ 
\fi
\end{definition}

\begin{lemma} \label{intocanv}
Let $f_\alpha(x)$ be  a family of operator concave functions on a unital JB-algebra $\Al$ 
indexed by $\alpha$ in $(0, \infty)$. Assume $f_\alpha(x)$ is
 uniformly Riemann integrable on $ \alpha  \in (0, \infty)$ for $x$ on 
 bounded and closed intervals. 
 %for each $x$, $f_\alpha(x)$ is Riemann integrable on $(0, \infty)$, 
 %and $f_\alpha(x)$ is continuous in $\alpha$ 
%uniformly with respect to $x$ on closed bounded intervals. 
%and for any positive number $M >0$ and positive element $A$ 
%in a JB-algebra ${\mathcal{A}}$ with $\| A \|| \leq M$, $f_\alpha(A)$ is continuous in $\alpha$ uniformly with respect to $A$.  
Then $\displaystyle \int_0^\infty f_\alpha(x) d \alpha$ is also operator concave. 
\end{lemma}
	
\begin{proof} For positive number $M>0$, it can be deduced from definition \ref{RiemannInt} that 
		there exist sequences $\delta_n \to 0$,  $ N_n \to \infty$ 
		and a sequence of equisubdivisions 
		\begin{align*}
			\delta_n = \alpha_{0(n)} <\alpha_{1(n)} <\cdots< \alpha_{m(n)}= N_n ,	
		\end{align*}
		of $[\delta_n, N_n]$ such that

	\begin{align}
	\int_0^\infty f_\alpha(x) d \alpha  = \lim_{n \rightarrow \infty } 
	\sum_{k=1}^{m(n)} f_{\alpha_{n_k}}(x)  \dfrac{N_n - \delta_n }{m(n)}	\; \; \ \text{uniformly on} \; x
	 \in [0, M]. \label{intunicon}
	\end{align}

The lemma follows from the following facts: (1)	
Linear combinations of operator concave functions with positive coefficients are operator concave. 
(2) The limit of such linear combinations is also operator concave.  
\end{proof}

\begin{remark} \label{rmkintunif}
When $x$ in Lemma {\rm \ref{intocanv}} 
is substituted by an element $A$ of a JB-algebra, 
the convergence 
in {\rm (\ref{intunicon})} 
is in norm by functional calculus for JB-algebras, \cite[Proposition 1.21]{Alfsen2003}.

\end{remark}
\begin{proposition} \label{LJocanv}
Let $\Al$ be a unital JB-algebra. The functions $x\mapsto x^{\lambda},$ $\lambda\in[0, 1],$ and $x\mapsto \log(x)$ are all operator concave on $(0,+\infty).$ 	
\end{proposition}
\begin{proof}
For each $s\in[0,1],$
\begin{align}\label{hai}
\left(s 1+(1-s) x\right)^{-1}\leq s 1+(1-s)x^{-1}.	
\end{align}
holds for any $x>0.$
Apply Theorem \ref{tnapf} to (\ref{hai}) with $h(t)=t,$ we derive that for positive invertible elements $A,$ $B$ in $\Al,$ and any $s\in[0,1],$ 
\begin{align*}
\left\{A^{-\frac{1}{2}}\left(sI+(1-s)\{A^{-\frac{1}{2}}B A^{-\frac{1}{2}}\}\right)^{-1}A^{-\frac{1}{2}}\right\}
\leq \left\{A^{-\frac{1}{2}} 
\left (sI+(1-s)\{A^{-\frac{1}{2}}B A^{-\frac{1}{2}}\}^{-1}\right)A^{-\frac{1}{2}}\right\}.	
\end{align*}
By Proposition \ref{3inv}, one has
\begin{align*}
\left(s A+(1-s) B\right)^{-1}&=\left\{A^{-\frac{1}{2}}\left(sI+(1-s)\{A^{-\frac{1}{2}}B A^{-\frac{1}{2}}\}\right)^{-1}A^{-\frac{1}{2}}\right\},\\
s A^{-1}+(1-s)B^{-1}&=\left\{A^{-\frac{1}{2}} \left(sI+(1-s)\{A^{-\frac{1}{2}}B A^{-\frac{1}{2}}\}^{-1}\right)A^{-\frac{1}{2}}\right\}.	
\end{align*}
Therefore,
\begin{align}\label{Jhai}
\left(s A+(1-s) B\right)^{-1}
\leq s A^{-1}+(1-s)B^{-1}.	
\end{align}
This shows that the function $x\mapsto x^{-1}$ is operator convex on $(0, \infty).$
Therefore, for each $\alpha>0,$ the function $f_\alpha(x) := (1+\alpha x)^{-1}x\alpha^{-\lambda}=\dfrac{1-(1+\alpha x)^{-1}}{\alpha} \alpha^{-\lambda}$ 
is operator concave on $(0, \infty).$ 

From \cite[Proposition 2.1]{RAISSOULI2017687}, we know that for $\lambda \in (0, 1)$ and $x \in (0, \infty)$, 
\begin{align}
x^{\lambda}=\dfrac{\sin(\lambda \pi)}{\pi}\int_0^{\infty} t^{\lambda -1}(1+tx^{-1})^{-1}dt.	
\label{powerf0}
\end{align}
By changing of variable $t=\frac{1}{\alpha},$ we have
\begin{align}
x^{\lambda}&=\dfrac{\sin(\lambda \pi)}{\pi}\int_0^{\infty} (1+\alpha x)^{-1}x \cdot\alpha^{-\lambda} d\alpha \nonumber\\ 
&=\frac{\sin(\lambda \pi)}{\pi}  \int_0^{\infty} f_{\alpha}(x)  d\alpha.	
\label{powerf}
\end{align}
We now show that $f_\alpha(x)$ 
is uniformly Riemann integrable on $ \alpha  \in (0, \infty)$ for $x$ on 
bounded and closed intervals, therefore Lemma \ref{intocanv} applies and  $x^{\lambda}$ is operator concave. In fact, 
\begin{align*}
	\left|\int_{\delta_1}^{\delta_2} f_\alpha(x) d \alpha \right|
	& =\left|\int_{\delta_1}^{\delta_2} (1+\alpha x)^{-1}x\alpha^{-\lambda} d \alpha \right| \\
	&\leq M \left|\int_{\delta_1}^{\delta_2} \alpha^{-\lambda} d \alpha \right|\to 0 \,\ \,\ {\rm uniformly}, \,\ \,\ {\rm as}\,\ \delta_1, \delta_2 \to 0.		\\
	\left|\int_{N_1}^{N_2} f_\alpha(x) d \alpha \right|
	& =\left|\int_{N_1}^{N_2} (1+\alpha x)^{-1}\alpha x\alpha^{-(\lambda+1)} d \alpha \right| \\
	&\leq  \left|\int_{N_1}^{N_2}\alpha^{-(\lambda+1)} d \alpha \right|\to 0 \,\ \,\ {\rm uniformly},  \,\ {\rm as}\,\ N_1, N_2 \to \infty.		
\end{align*}

Let $C([0,M])$ be the Banach space of continuous functions on $[0, M]$.  
For fixed interval $[a, b] \subset (0, \infty)$ 
%and   each $\alpha \in [a, b]$, define $h(\alpha) \in C([0,M])$ 
define 
$h: \alpha\in[a, b] \mapsto h(\alpha) \in C([0,M])$
by 
 $h(\alpha)(x): = f_\alpha (x)$. Then $ h(\alpha)$ is continuous on $[a, b]$
 because
  \begin{align*}
 	\| h_\alpha - h_{\alpha_0} \|& =
 	\sup_{ x \in [0,M] } \left|
 	\dfrac{ x}{(1+\alpha x)\alpha^{\lambda}} - 	\dfrac{ x}{(1+\alpha_0 x)\alpha_0^{\lambda}} \right|\\ 
 &\leq \frac{M}{a^{2\lambda}}\sup_{ x \in [0,M] } \left|
(1+\alpha x)\alpha^{\lambda} - (1+\alpha_0 x)\alpha_0^{\lambda} \right|	\to 0, \,\ \,\ {\rm as} \,\ \alpha\to \alpha_0.
 \end{align*}
 Then $C([0,M])$-valued Riemann integral 
	$\displaystyle \int_a^b h(\alpha) d \alpha$ exists, which implies (\ref{Riemm2}).

Denote $g_{\alpha}(x)=(\alpha+1)^{-1}-(\alpha+x)^{-1}.$ By (\ref{Jhai}), $g_{\alpha}(x)$ is operator concave for $\alpha\geq 0.$

For any $x>0,$ 
\begin{align}
\log(x)=\int_0^1 \dfrac{1-(1-t+tx)^{-1}}{t}dt,	
\end{align}
which could be derived in the proof of \cite[Proposition 3.1]{RAISSOULI2017687}.
Replacing $t$ by $\dfrac{1}{\alpha+1},$ $\alpha\in(0,\infty),$
\begin{align}
\log(x)&=\int_0^{\infty} [(\alpha+1)^{-1}-(\alpha+x)^{-1}]d\alpha 
\nonumber 
\\
&=\int_0^{\infty} g_{\alpha}(x)d\alpha. 
\label{logf}
\end{align}
By the same reasoning as for $x^\lambda$ above, $\log(x)$ is operator concave.
\end{proof}
%\begin{remark}
%For any $\lambda \in [0,1],$ the function $x^{\lambda}$ is 	operator monotone increasing on $(0,+\infty)$ {\rm(}see eg. \cite[Lemma 3.1]{NEAL2000284}{\rm)}. Moreover, $x^{-1}$ is operator monotone decreasing \cite[Proposition 3.5.3]{hanche1984jordan}, then  $g_{\alpha}(x)$ is operator monotone increasing. Since operator monotonicity is preserved under limits and convex combination, then $\log(x)$ is also operator monotone increasing. 
%\end{remark}

Similar argument as in Lemma \ref{intocanv} gives 

\begin{lemma} \label{intomono}
Let $f_\alpha(x)$ be  a family of operator monotone functions on a unital JB-algebra $\Al$ 
indexed by $\alpha$ in $(0, \infty)$. Assume $f_\alpha(x)$ is
uniformly Riemann integrable on $ \alpha  \in (0, \infty)$ for $x$ on 
bounded and closed intervals. 
Then $\displaystyle \int_0^\infty f_\alpha(x) d \alpha$ is also operator monotone. 

\end{lemma}

\begin{remark} \label{NewRes}
Lemmas {\rm \ref{intocanv}} and {\rm \ref{intomono}} are also true for 
operators on Hilbert space because $B(H)_{sa}$ is a special 
JB-algebra. The two results do not seem to appear in 
literature on Hilbert space operators. 
\end{remark}

\begin{proposition}\label{opmpowerlog}
Let $\Al$ be a unital JB-algebra. The functions $x\mapsto x^{\lambda},$ $\lambda\in[0, 1],$ and $x\mapsto \log(x)$ are all operator monotone increasing on $(0,+\infty).$ 	
\end{proposition}

\begin{proof}
	
	By \cite[Proposition 3.5.3]{hanche1984jordan}, 
	$x^{-1}$ is operator monotone decreasing. Thus, $f_{\alpha}(x),$ $g_{\alpha}(x)$ are operator monotone increasing. By Lemma \ref{intomono}, $x^{\lambda}$ and $\log(x)$ is also operator monotone increasing. 
	
	Note that the monotonicity of $x^{\lambda}$ has another proof in \cite[Lemma 3.1]{NEAL2000284} 
	using different techniques.

\end{proof}

\section{Operator means}

In this section, 
we introduce operator means for elements in JB-algebras and establish some basic identities and 
inequalities.

For two positive invertible elements $A, B$ in a unital JB-algebra $\Al$ and $0\leq \lambda\leq 1$, we denote
\begin{align*}
&\mbox{the weighted harmonic mean}: A!_{\lambda} B=\left((1-\lambda)A^{-1}+\lambda B^{-1}\right)^{-1};\\
&\mbox{the weighted geometric mean}: A\#_{\lambda} B=\left\{A^{\frac{1}{2}}\{A^{-\frac{1}{2}}BA^{-\frac{1}{2}}\}^{\lambda}A^{\frac{1}{2}}\right\};\\
&\mbox{the weighted arithmetic mean}: A\triangledown_{\lambda} B=(1-\lambda)A+\lambda B.	
\end{align*}

The associative algebra version of the following theorem and corollary 
for operators on Hilbert space are proved in \cite{RAISSOULI2017687}. 
Here we establish the non-associative version which has different connotations than 
the associative version using the theory of JB-algebras.
 
\begin{theorem}\label{TJwgitg}
Let $A, B$ be two positive invertible elements in a unital JB-algebra $\Al.$ Then for any $0<\lambda <1$
\begin{align}
A\#_{\lambda}B
&=\dfrac{\sin(\lambda \pi)}{\pi}\int_0^{\infty} t^{\lambda -1}(A^{-1}+t B^{-1})^{-1}dt	 \label{Lgnitg1} \\
&=\dfrac{\sin(\lambda \pi)}{\pi}\int_0^1 \frac{t^{\lambda -1}}{(1-t)^{\lambda}}(A!_t B)dt.
\end{align}
\end{theorem}

\begin{proof}
By  Proposition \ref{3inv}, 
\begin{align*}
&B=U_{A^{-\frac{1}{2}}}U_{A^{\frac{1}{2}}}(B)
=\left\{A^{-\frac{1}{2}}\{A^{\frac{1}{2}}BA^{\frac{1}{2}}\}A^{-\frac{1}{2}}\right\}		\\
&\left(U_{A^{1/2}}(B)\right)^{-1}
=\{A^{\frac{1}{2}}BA^{\frac{1}{2}}\}^{-1}=\{A^{-\frac{1}{2}}B^{-1}A^{-\frac{1}{2}}\}.	
\end{align*}

Denote
\begin{align*}
E=\dfrac{\sin(\lambda \pi)}{\pi}\int_0^{\infty} t^{\lambda -1}(A^{-1}+tB^{-1})^{-1}dt.	
\end{align*}

Therefore,
\begin{align}
E
&=\dfrac{\sin(\lambda \pi)}{\pi}\int_0^{\infty} t^{\lambda -1}\left[U_{A^{-\frac{1}{2}}}(1+tU_{A^{\frac{1}{2}}}B^{-1})\right]^{-1}dt \nonumber \\
&=\dfrac{\sin(\lambda \pi)}{\pi}\int_0^{\infty} t^{\lambda -1}\left\{A^{-\frac{1}{2}}\left(1+t\{A^{\frac{1}{2}}B^{-1}A^{\frac{1}{2}}\}\right)A^{-\frac{1}{2}}\right\}^{-1}dt \nonumber\\
&=\dfrac{\sin(\lambda \pi)}{\pi}\int_0^{\infty} t^{\lambda -1} \left\{A^{\frac{1}{2}} \left(1+t\{A^{\frac{1}{2}}B^{-1}A^{\frac{1}{2}}\}\right)^{-1}	A^{\frac{1}{2}}\right\}dt \nonumber\\
&=\left\{A^{\frac{1}{2}} \dfrac{\sin(\lambda \pi)}{\pi} \int_0^{\infty} t^{\lambda -1} \left(1+t\{A^{-\frac{1}{2}}BA^{-\frac{1}{2}}\}^{-1}\right)^{-1}dt~A^{\frac{1}{2}}\right\}.  \nonumber
\end{align}

Applying functional calculus in JB-algebras (see e.g. \cite[Proposition 1.21]{Alfsen2003}) 
to (\ref{powerf0}),   
\begin{align*}
\left\{A^{-\frac{1}{2}}BA^{-\frac{1}{2}}\right\}^{\lambda}
=\dfrac{\sin(\lambda \pi)}{\pi}\int_0^{\infty} t^{\lambda -1} \left(1+t\{A^{-\frac{1}{2}}BA^{-\frac{1}{2}}\}^{-1}\right)^{-1}dt.	
\end{align*}

Then,
\begin{align*}
A\#_{\lambda} B&=\left\{A^{\frac{1}{2}}\{A^{-\frac{1}{2}}BA^{-\frac{1}{2}}\}^{\lambda}A^{\frac{1}{2}}\right\}\\
&=\left\{A^{\frac{1}{2}}\dfrac{\sin(\lambda \pi)}{\pi}\int_0^{\infty} t^{\lambda -1} \left(1+t\{A^{-\frac{1}{2}}BA^{-\frac{1}{2}}\}^{-1}\right)^{-1}dt~A^{\frac{1}{2}}\right\}\\
&=\frac{\sin(\lambda \pi)}{\pi}\int_0^{\infty} t^{\lambda -1}(A^{-1}+tB^{-1})^{-1}dt.
\end{align*}

%\begin{align*}
%A\#_{\lambda} B
%=\left\{A^{\frac{1}{2}}\{A^{-\frac{1}{2}}BA^{-\frac{1}{2}}\}^{\lambda}A^{\frac{1}{2}}\right\}
%=\dfrac{\sin(\lambda \pi)}{\pi}\int_0^{\infty} t^{\lambda -1}(A^{-1}+tB^{-1})^{-1}dt.	
%\end{align*}

By changing of variable $t=\dfrac{\gamma}{1-\gamma},$ $\gamma\in (0,1)$ in (\ref{Lgnitg1}), we have 
\begin{align*}
A\#_{\lambda} B=\dfrac{\sin(\lambda \pi)}{\pi}\int_0^1 \frac{t^{\lambda -1}}{(1-t)^{\lambda}}(A!_t B)dt.
\end{align*}
\end{proof}

\begin{corollary}\label{Jgmasym}
For any any positive invertible elements $A, B$ in $\Al$ and $\lambda \in [0,1],$ then $A\#_{\lambda} B=B\#_{1-\lambda} A,$ i.e., 
$$\left\{A^{\frac{1}{2}}\{A^{-\frac{1}{2}}BA^{-\frac{1}{2}}\}^{\lambda}A^{\frac{1}{2}}\right\}
=\left\{B^{\frac{1}{2}}\{B^{-\frac{1}{2}}AB^{-\frac{1}{2}}\}^{1-\lambda}B^{\frac{1}{2}}\right\}.$$ 
Moreover, if $\lambda=\frac{1}{2},$ then $A\# B=B\#A.$
\end{corollary}
\begin{proof}
	It is straightforward to verify the identity for $\lambda = 0, 1$. We assume 
	 $\lambda \neq 0, 1$. 
From Theorem \ref{TJwgitg}, 
\begin{align*}
A\#_{\lambda}B=\dfrac{\sin(\lambda \pi)}{\pi}\int_0^1 \frac{t^{\lambda -1}}{(1-t)^{\lambda}}(A!_t B)dt.	
\end{align*}
Replacing $t$ by $1-s$ with $s\in [0,1]$ we have
\begin{align*}
A\#_{\lambda}B
&=\dfrac{\sin\left((1-\lambda) \pi\right)}{\pi}\int_0^1\frac{(1-s)^{\lambda -1}}{s^{\lambda}}(A!_{1-s} B)ds\\
&=\dfrac{\sin\left((1-\lambda) \pi\right)}{\pi}\int_0^1\frac{s^{(1-\lambda)-1}}{(1-s)^{1-\lambda}}(B!_s A)ds\\
&=B\#_{1-\lambda} A.
\end{align*}

Moreover, if $\lambda =\dfrac{1}{2},$ then $A\# B=A\#_{\frac{1}{2}}B=B\#_{\frac{1}{2}}A=B\#A.$
\end{proof}

\begin{proposition}\label{pjwgn}
The weighted geometric mean $A\#_{\lambda}B$ defined in $\Al$ has the following properties:
\begin{itemize}
\item[(i)]  $(\alpha A)\#_{\lambda}(\beta B)= \alpha^{1-\lambda} \beta^{\lambda} (A\#_{\lambda}B),$ for any nonnegative numbers $\alpha$ and $\beta.$ 
\item[(ii)] If $A\leq C$ and $B\leq D,$ then $A\#_{\lambda}B\leq C\#_{\lambda}D.$	
\item[(iii)] $A\#_{\lambda}B$ is operator concave with respect to $A,B$ individually. 
\item[(iv)] $\{C(A\#_{\lambda}B)C\}=\{CAC\}\#_{\lambda}\{CBC\},$ for any invertible $C$ in $\Al.$
\item[(v)] $(A\#_{\lambda} B)^{-1}=A^{-1}\#_{\lambda} B^{-1}.$
\end{itemize}
\end{proposition}
\begin{proof}
For (i), it follows directly from the definition. 

Proof of (ii). Since it is trivial for $\lambda =0$ or $\lambda =1,$ here we only give the proof for $0<\lambda<1.$ If $B\leq D,$ then 
$$\{A^{-\frac{1}{2}}BA^{-\frac{1}{2}}\}\leq \{A^{-\frac{1}{2}}DA^{-\frac{1}{2}}\}.$$	 
By Corollary \ref{opmpowerlog}, the following inequality
$$\{A^{-\frac{1}{2}}BA^{-\frac{1}{2}}\}^{\lambda}\leq \{A^{-\frac{1}{2}}DA^{-\frac{1}{2}}\}^{\lambda},$$
holds for any $\lambda \in[0,1].$ 
Therefore,
\begin{align}\label{incJwgm1}
A\#_{\lambda}B
=\left\{A^{\frac{1}{2}}\{A^{-\frac{1}{2}}BA^{-\frac{1}{2}}\}^{\lambda}A^{\frac{1}{2}}\right\}
\leq \left\{A^{\frac{1}{2}}\{A^{-\frac{1}{2}}DA^{-\frac{1}{2}}\}^{\lambda}A^{\frac{1}{2}}\right\}
=A\#_{\lambda}D.
\end{align}

Similarly, if $A\leq C,$ then 
\begin{align}\label{incJwgm2}
A\#_{\lambda}D=D\#_{1-\lambda}A
\leq D\#_{1-\lambda}C
=C\#_{\lambda}D.		
\end{align}
Combing (\ref{incJwgm1}) and (\ref{incJwgm2}), we obtain the desired result.

(iii) For any $0\leq t\leq 1,$ by Proposition \ref{LJocanv}, we have 
\begin{align*}
A\#_{\lambda} [(1-t)B_1+tB_2]&=\left\{A^{\frac{1}{2}}\{A^{-\frac{1}{2}}[(1-t)B_1+tB_2]A^{-\frac{1}{2}}\}^{\lambda}A^{\frac{1}{2}}\right\}\\
&=\left\{A^{\frac{1}{2}}\left[(1-t)\{A^{-\frac{1}{2}}B_1A^{-\frac{1}{2}}\}+t\{A^{-\frac{1}{2}}B_2A^{-\frac{1}{2}}\}\right]^{\lambda}A^{\frac{1}{2}}\right\}	\\
&\geq (1-t)\left\{A^{\frac{1}{2}}\{A^{-\frac{1}{2}}B_1A^{-\frac{1}{2}}\}^{\lambda}A^{\frac{1}{2}}\right\}+t\left\{A^{\frac{1}{2}}\{A^{-\frac{1}{2}}B_2 A^{-\frac{1}{2}}\}^{\lambda}A^{\frac{1}{2}}\right\}\\
&=(1-t)A\#_{\lambda} B_1+tA\#_{\lambda} B_2.
\end{align*}
Similarly, one can show that $B\#_{1-\lambda}A$ is operator concave with respect to $A.$ Since $A\#_{\lambda}B=B\#_{1-\lambda}A,$ then $A\#_{\lambda}B$ is also operator concave with respect to $A.$

Proof of (iv).  According to Theorem \ref{TJwgitg} and Proposition \ref{3inv}
\begin{align*}
\{CAC\}\#_{\lambda} \{CBC\}
&=\dfrac{\sin(\lambda \pi)}{\pi} \int_0^{\infty} t^{\lambda -1} \left(\{CAC\}^{-1}+t\{CBC\}^{-1}\right)^{-1}dt \\
&=\dfrac{\sin(\lambda \pi)}{\pi} \int_0^{\infty} t^{\lambda -1} \{C^{-1}(A^{-1}+tB^{-1})C^{-1}\}^{-1}dt\\
&=\dfrac{\sin(\lambda \pi)}{\pi} \int_0^{\infty} t^{\lambda -1} \{C(A^{-1}+tB^{-1})^{-1}C\}dt\\
&=\left\{C \dfrac{\sin(\lambda \pi)}{\pi} \int_0^{\infty} t^{\lambda -1} (A^{-1}+tB^{-1})^{-1}dt~C\right\}\\
&=\{C (A\#_{\lambda}B)C\}.
\end{align*}

(v) According to Lemma 3.2.10 and  Proposition \ref{3inv}
\begin{align*}
(A\#_{\lambda} B)^{-1}&=\left\{A^{\frac{1}{2}}\left(\{A^{-\frac{1}{2}}BA^{-\frac{1}{2}}\}\right)^{\lambda}A^{\frac{1}{2}}\right\}	^{-1}\\
&=\left\{A^{-\frac{1}{2}}\left(\{A^{-\frac{1}{2}}BA^{-\frac{1}{2}}\}\right)^{-\lambda}A^{-\frac{1}{2}}\right\}\\
&=\left\{A^{-\frac{1}{2}}\left(\{A^{\frac{1}{2}}B^{-1}A^{\frac{1}{2}}\}\right)^{\lambda}A^{-\frac{1}{2}}\right\}\\
&=A^{-1}\#_{\lambda} B^{-1}.
\end{align*}
\end{proof}

It is well-known that the following Young inequalities 
\begin{align} \label{owhga}
A!_{\lambda} B
\leq A^{\frac{1}{2}}(A^{-\frac{1}{2}}BA^{-\frac{1}{2}})^{\lambda}A^{\frac{1}{2}}
\leq A\triangledown_{\lambda} B	
\end{align}	
hold for any strictly positive operators $A$ and $B$ on complex Hilbert space $H.$ 

The next theorem generalizes (\ref{owhga}) to JB-algebras.
\begin{theorem}[Young inequalities for JB-algebras]\label{LJwhga}
Let $A, B$ be positive invertible elements in $\Al.$ For any $0\leq \lambda \leq 1,$ 
\begin{align} \label{LJwhga2}
A!_{\lambda} B
\leq A\#_{\lambda} B
\leq A\triangledown_{\lambda} B.	
\end{align}	
\end{theorem}
\begin{proof}
By Proposition \ref{3inv},  
we have 
\begin{align*}
A!_{\lambda} B
&=\left\{A^{\frac{1}{2}} \left((1-\lambda)1+\lambda \{A^{-\frac{1}{2}}BA^{-\frac{1}{2}}\}^{-1}\right)^{-1}A^{\frac{1}{2}} \right\},\\
A\triangledown_{\lambda} B
&=\left\{A^{\frac{1}{2}} \left((1-\lambda)+\lambda \{A^{-\frac{1}{2}}BA^{-\frac{1}{2}}\}\right)A^{\frac{1}{2}}\right\}.
\end{align*}
Let
\begin{align}
r(x)&=\left[(1-\lambda)1+\lambda x^{-1}\right]^{-1},\\
q(x)&=x^{\lambda},\\
k(x)&=(1-\lambda)1+\lambda x	.	
\end{align}
One sees that for any $0\leq \lambda \leq 1,$
\begin{align}\label{Jhgai2}
r(x)\leq q(x)\leq k(x).
\end{align}
hold for all $x>0.$ 
Denoting $h(t)=t,$ we have
\begin{align}
P_{r\triangle h}(B,A)&=\left\{A^{\frac{1}{2}} \left((1-\lambda)1+\lambda \{A^{-\frac{1}{2}}BA^{-\frac{1}{2}}\}^{-1}\right)^{-1}A^{\frac{1}{2}} \right\},\nonumber\\
P_{q\triangle h}(B,A)&=\left\{A^{\frac{1}{2}} \left(\{A^{-\frac{1}{2}}BA^{-\frac{1}{2}}\}\right)^{\lambda}A^{\frac{1}{2}} \right\},\label{JYoung}\\	
P_{k\triangle h}(B,A)&=\left\{A^{\frac{1}{2}} \left((1-\lambda)+\lambda \{A^{-\frac{1}{2}}BA^{-\frac{1}{2}}\}\right)A^{\frac{1}{2}}\right\}.\nonumber
\end{align}
Applying Theorem \ref{tnapf} to (\ref{JYoung}), we derive the inequalities
\begin{align*} 
A!_{\lambda} B
\leq A\#_{\lambda} B
\leq A\triangledown_{\lambda}B.	
\end{align*}	
\end{proof}
As an improvement of (\ref{LJwhga2}), we have the following refined Young inequalities, which has origin in \cite{furuichi2012refined} for operators on Hilbert space. 
\begin{proposition}\label{LJwhga3}
Let $A, B$ be positive invertible elements in $\Al.$ For any $0\leq \lambda \leq 1$ and $\delta=\min\{ \lambda, 1-\lambda\}$
\begin{align} \label{LJwhga4}
A!_{\lambda} B
&\leq \left[A^{-1}\#_{\lambda}B^{-1}+2\delta\left(\frac{A^{-1}+B^{-1}}{2}-A^{-1}\#_{1/2}B^{-1}\right)\right]^{-1} \\
&\leq A\#_{\lambda} B\\
&\leq A\#_{\lambda} B+2\delta\left(\frac{A+B}{2}-A\#_{1/2}B\right)\\
&\leq A\triangledown_{\lambda} B.	
\end{align}	
\end{proposition}
\begin{proof}
By Theorem \ref{LJwhga}, $\dfrac{A+B}{2}-A\#_{1/2}B\geq 0.$ This implies that 
$$A\#_{\lambda} B\leq A\#_{\lambda} B+2\delta\left(\frac{A+B}{2}-A\#_{1/2}B\right).$$
Similarly, 
\begin{align}\label{Jwgmie}
A^{-1}\#_{\lambda} B^{-1}\leq A^{-1}\#_{\lambda} B^{-1}+2\delta\left(\frac{A^{-1}+B^{-1}}{2}-A^{-1}\#_{1/2}B^{-1}\right).	
\end{align}

Taking inverses in  (\ref{Jwgmie}) and then applying Proposition \ref{pjwgn}.(v) gives 
$$\left[A^{-1}\#_{\lambda} B^{-1}+2\delta\left(\frac{A^{-1}+B^{-1}}{2}-A^{-1}\#_{1/2}B^{-1}\right)\right]^{-1}\leq A\#_{\lambda} B.$$
By \cite[Theorem 2.1]{KITTANEH2010262}, for $x\geq 0$ and $0\leq \lambda \leq 1,$
\begin{align}\label{hga1}
x^{\lambda}+ 2\delta \left( \dfrac{x+1}{2}- \sqrt{x} \right) \leq (1-\lambda)+\lambda x	.
\end{align}

Applying Theorem \ref{tnapf} to inequality (\ref{hga1}) with $h(t) = t$ gives
$$A\#_{\lambda} B+2\delta\left(\frac{A+B}{2}-A\#_{1/2}B\right) \leq A\triangledown_{\lambda} B.$$
Similarly,
\begin{align}
A^{-1}\#_{\lambda} B^{-1}+2\delta\left(\frac{A^{-1}+B^{-1}}{2}-A^{-1}\#_{1/2}B^{-1}\right) \leq A^{-1}\triangledown_{\lambda} B^{-1}.
\label{xdelta}
\end{align}
By definition,  
\begin{align*}
(A^{-1}\triangledown_{\lambda} B^{-1})^{-1}=A!_{\lambda} B,	
\end{align*}
which combined with (\ref{xdelta}) gives
$$A!_{\lambda} B
\leq \left[A^{-1}\#_{\lambda}B^{-1}+2\delta\left(\frac{A^{-1}+B^{-1}}{2}-A^{-1}\#_{1/2}B^{-1}\right)\right]^{-1}.$$
\end{proof}

Propositions \ref{KuboAndoT1} and 
\ref{KuboAndoT2} below have origins in the classical 
Kubo-Ando theory.  

\begin{proposition}
	\label{KuboAndoT1}
Let $0\leq \delta\leq 1$ and $A, B$ be positive invertible elements in $\Al$. For $0<\lambda<1,$ 
\begin{align}\label{Jhgai3}
\delta A\#_{\lambda}B+(1-\delta)A\triangledown_{\lambda}B
\geq A!_{\lambda}B.
\end{align}		
\end{proposition}
\begin{proof}
For all $x>0,$ the inequality 
\begin{align}\label{hgai3}
\delta x^{\lambda}+(1-\delta)[(1-\lambda)+\lambda x] \geq[(1-\lambda)+\lambda x^{-1}]^{-1}.
\end{align}
can be found in \cite[Lemma 2.6]{Furuichi2014operator}.
Using the perspective functions associated with these two functions in (\ref{hgai3}) with $h(t)=t,$ and applying Theorem \ref{tnapf},  the desired result follows.		
\end{proof}

\begin{proposition}
\label{KuboAndoT2}
Suppose $\delta\geq 2$ and $A, B$ are positive invertible elements in a unital JB-algebra $\Al$. If $0\leq \lambda\leq  \frac{1}{2}$ and $0< A\leq B,$ or $\frac{1}{2}\leq \lambda\leq 1 $ and $B\leq A,$ then
\begin{align}\label{Jhgai2}
\delta A\#_{\lambda}B+(1-\delta)A\triangledown_{\lambda}B
\leq A!_{\lambda}B.
\end{align}	
\end{proposition}
\begin{proof}
From Lemma 2.5 \cite{Furuichi2014operator} , we know that 
\begin{align}\label{hgai2}
\delta x^{\lambda}+(1-\delta)[(1- \lambda)+ \lambda x] 
\leq[(1-\lambda)+\lambda x^{-1}]^{-1}.
\end{align}	
Appling Theoren \ref{tnapf} to the inequalities (\ref{hgai2}) with $h(t)=t$ gives (\ref{Jhgai2}).
\end{proof}

For any positive number $x\in [\frac{\alpha}{\beta}, \frac{\beta}	{\alpha}]$  and $\lambda \in[0,1],$ we have the following inequalities  
\begin{align}\label{Spri}
% S(x^r) x^{\lambda} \leq 
x^{\lambda} \leq 
(1-\lambda) +\lambda x\leq S(x) x^{\lambda},	
\end{align}
where 
 $S(x)$ is Specht's ratio (see e.g. \cite[Lemma 2.3]{tominaga2002specht}) 
 and its graph is like a parabola with minimum value $S(1)=1$. 
By \cite[Lemma 1]{furuichi2012refined}, 
\begin{align}\label{mvspr}
\sup_{x\in [\frac{\alpha}{\beta},\frac{\beta}{\alpha}]}S(x)
=\max\left\{S \left( \frac{\alpha}{\beta} \right), 
S \left( \frac{\beta}{\alpha} \right) \right\}= S \left( \frac{\beta}{\alpha} \right).	
\end{align}

The following result is a refined Young inequalities with Specht's ratio in the setting of JB-algebras.

\begin{proposition}
Let $A, B$ be positive invertible elements in a unital JB-algebra such that for 
positive numbers $\alpha$ and $\beta$, $\alpha \leq A \leq \beta$ and $\alpha \leq B\leq \beta.$ 
Then 
\begin{align*}
A\#_{\lambda} B\leq A\triangledown_{\lambda} B\leq S \left( \frac{\beta}{\alpha} \right) A\#_{\lambda} B	
\end{align*}
hold for any $\lambda\in[0,1].$	
\end{proposition}
\begin{proof}
According to \cite[Lemma 3.5.3]{hanche1984jordan}, 
$$\beta^{-1}\leq  A^{-1} \leq \alpha^{-1} \; \; \text{and} \; \;  
\beta^{-1}\leq B^{-1}\leq \alpha^{-1}.$$
 Moreover,
\begin{align}
&\frac{\alpha}{\beta}I \leq \alpha A^{-1}=U_{{A^{-\frac{1}{2}}}}(\alpha I) \leq U_{{A^{-\frac{1}{2}}}}(B)=\{A^{-\frac{1}{2}} B A^{-\frac{1}{2}}\} \label{JBi1},	\\
&\frac{\alpha}{\beta}I\leq \beta^{-1}A=U_{{A^{\frac{1}{2}}}}(\beta^{-1} I)\leq U_{{A^{\frac{1}{2}}}}(B^{-1})=\{A^{-\frac{1}{2}} B A^{-\frac{1}{2}}\}^{-1}. \label{JBi2}
\end{align}
Combining (\ref{JBi1}) and (\ref{JBi2}), we have
\begin{align}
\frac{\alpha}{\beta}I\leq \{A^{-\frac{1}{2}} B A^{-\frac{1}{2}}\} \leq \frac{\beta}	{\alpha}I.
\end{align}

Since $S(x) x^{\lambda} \leq S \left( \frac{\beta}{\alpha} \right)  x^{\lambda}$,
it follows from (\ref{Spri}) and (\ref{mvspr}) that
\begin{align}\label{Spri1}
x^{\lambda}\leq (1-\lambda) +\lambda x\leq S \left( \frac{\beta}{\alpha} \right) x^{\lambda}.
\end{align}
Applying Theorem \ref{tnapf} with $h(t) =t$ to (\ref{Spri1}), the desired inequalities follow.
\end{proof}

\bigskip
\noindent 
{\bf Acknowledgements} We are grateful to all the referees for their careful readings of 
our manuscript, comments,  and suggestions. These greatly helped to improve the quality and 
readability of the paper. 
We are especially indebted to the referee who pointed out the inadequacy of assumptions  
for Lemmas 1 and 2 in an early version of the manuscript.

%\section{Declarations (required by the journal)}
%{\bf Data availability:} 
%Data sharing is not applicable to this article as no new data were created or analyzed in this study.
%
%{\bf Competing interests:} 
%There are no competing interests in this paper.
%
%{\bf Funding:}
%The authors did not receive funding to support this research. 
%
%{\bf Authors' contributions:}
%Each author contributed equally in this paper. 
%
%{\bf Acknowledgements:} None.

\end{document}